\documentclass[preprint,authoryear,5p,times]{elsarticle}
\usepackage{amssymb}
\usepackage{float}
\usepackage{amsmath}
\usepackage{amssymb}
\usepackage{amsthm}
\usepackage{mathrsfs}
\usepackage[mathscr]{euscript}
\usepackage[ruled,vlined]{algorithm2e}
\usepackage{arydshln}
\usepackage{multirow}
\usepackage{multicol}
\usepackage{lineno,hyperref}

\usepackage{xcolor}
\hypersetup{
    colorlinks,
    linkcolor={blue},
    citecolor={blue},
    urlcolor={blue}
}

\journal{Advances in Space Research}

\newtheorem{theorem}{Theorem}[section]

\newtheorem{assumption}[theorem]{Assumption}
\newtheorem{proposition}[theorem]{Proposition}

\theoremstyle{remark}
\newtheorem{remark}[theorem]{Remark}
\theoremstyle{definition}

\begin{document}

\begin{frontmatter}

\title{Multiple-impulse orbital maneuver with limited observation window}

\author{Amir Shakouri\fnref{label01}}
\address{Department of Aerospace Engineering, Sharif University of Technology, Tehran 14588-89694, Iran}
\fntext[label01]{Corresponding author. Email addresses: \href{mailto:a_shakouri@outlook.com}{a$\_$shakouri@outlook.com} (A. Shakouri), \href{mailto:pourtak@sharif.edu}{pourtak@sharif.edu} (S. H. Pourtakdoust), \href{mailto:msayanjali@gmail.com}{msayanjali@gmail.com} (M. Sayanjali).}
\author{Seid H. Pourtakdoust}
\address{Center for Research and Development in Space Science and Technology, Sharif University of Technology, Tehran 14588-89694, Iran}
\author{Mohammad Sayanjali}
\address{Satellite Research Institute, Tehran 19979-94313, Iran}

\begin{abstract}

This paper proposes a solution for multiple-impulse orbital maneuvers near circular orbits for special cases where orbital observations are not globally available and the spacecraft is being observed through a limited window from a ground or a space-based station. The current study is particularly useful for small private launching companies with limited access to global observations around the Earth and/or for orbital maneuvers around other planets for which the orbital observations are limited to the in situ equipment. An appropriate cost function is introduced for the sake of minimizing the total control/impulse effort as well as the orbital uncertainty. It is subsequently proved that for a circle-to-circle maneuver, the optimization problem is quasi-convex with respect to the design variables. For near circular trajectories the same cost function is minimized via a gradient based optimization algorithm in order to provide a sub-optimal solution that is efficient both with respect to energy effort and orbital uncertainty. As a relevant case study, a four-impulse orbital maneuver between circular orbits under Mars gravitation is simulated and analyzed to demonstrate the effectiveness of the proposed algorithm.

\end{abstract}


\end{frontmatter}


\section{Introduction}
\label{S:1}

Impulsive orbital maneuvers (IOM) have always been a challenging issue in astrodynamics. In 1925 the well-known two-impulse maneuver for transfer between coplanar orbits was first proposed by Walter Hohmann \citep{in1} that was demonstrated to be an optimal solution for the unconstrained transfer problem. Subsequently, many researches have focused on IOM in order to propose and improve the solutions for more complex and challenging situations by considering different constraints on the problem \citep{in2,in3,in4,in5,in6,in7,in7p}, while some researches have tried to find more simplified methods to solve the multiple-impulse orbital maneuver (MIOM) with less computational efforts \citep{in8,in9,in10}. 

On the other hand, the Lambert's approach \citep{gooding1990procedure,albouy2019lambert} has traditionally been utilized to establish conic trajectories between any two spatial points in space within a predefined time interval that can be directly used as the maneuver trajectory, as well. Several solution methods are proposed in the literature to enhance the speed and accuracy of the early algorithms in which the reader can refer to \citep{leeghim2010energy,de2018solution,russell2019solution} and the references therein. Many enhanced versions over the classical Lambert's method have emerged since its original introduction due to its vast applicability for multiple-revolution \citep{in14,in15,in16}, perturbed \citep{in17,in18}, and optimized transfer solutions \citep{in19} in various related contexts. 

However, the majority of the state-of-the-art IOM methods do not consider the issue of realistic uncertainties such as measurement and process noise, actuation errors, etc. The inherent nature of the stochastic uncertainty could indeed affect the mission design (MD) parameters that are usually not considered at the initial MD stages. Of course there exist powerful estimation and filtering techniques that can compensate for the role of uncertainty with acceptable accuracy for offline MD when the system is fully observable. On the other hand, the uncertainty problem persists when there are lack of sufficient observability level and time, where state estimation will no longer produce a converged solution to an acceptable error bound. In these scenarios, the estimation uncertainty, modeled via the system covariance matrix, can be considered as an affecting tool in the process of MD for orbital maneuvers. The continuous-thrust two-dimensional coplanar orbital maneuvers under poor measurements are studied in \citep{in20} where the diagonal elements of the covariance matrix are augmented in the cost function to be minimized alongside the control effort. For spacecraft rendezvous, a same approach for the impulsive maneuvers under uncertainty is discussed in \citep{in21} where a multi-objective unconstrained optimization approach is implemented and analyzed. In \citep{in22} the multiple-impulse rendezvous problem is studied by proposing a covariance minimization approach under several constraints on the maximum control effort, maximum thruster limit, and maximum flight time while satisfying some safe approach corridors.

This current study investigates the MIOM problem under a special realistic assumption that the measurements are not available at all times and the state estimation cannot retain its convergence beyond the observation window (see Fig. \ref{fig:0p}). In this situation, thrust actuation cannot be performed in the blind regions for sure and nevertheless, the impulses should not be applied at the early times of entering the observation window as well, as obviously the state variables need a sufficient time for convergence trough the filtering process. Therefore, a cost function is introduced in which by its minimization the sum of impulses will be reduced and also enough time will be given to the system for the relaxation of its estimation errors. It is also demonstrated that the cost function is quasi-convex for two-impulse circle-to-circle maneuvers. For MIOMs, the transfer trajectory is approximated by several two-impulse circle-to-circle maneuvers and a gradient based optimization technique is implemented to establish a solution. The Lambert's algorithm is also used for trajectory generation and calculation of the cost function in which the impulse positions and impulse times are considered as the optimization (design) variables. 

\begin{figure}[h]
\centering\includegraphics[width=1\linewidth]{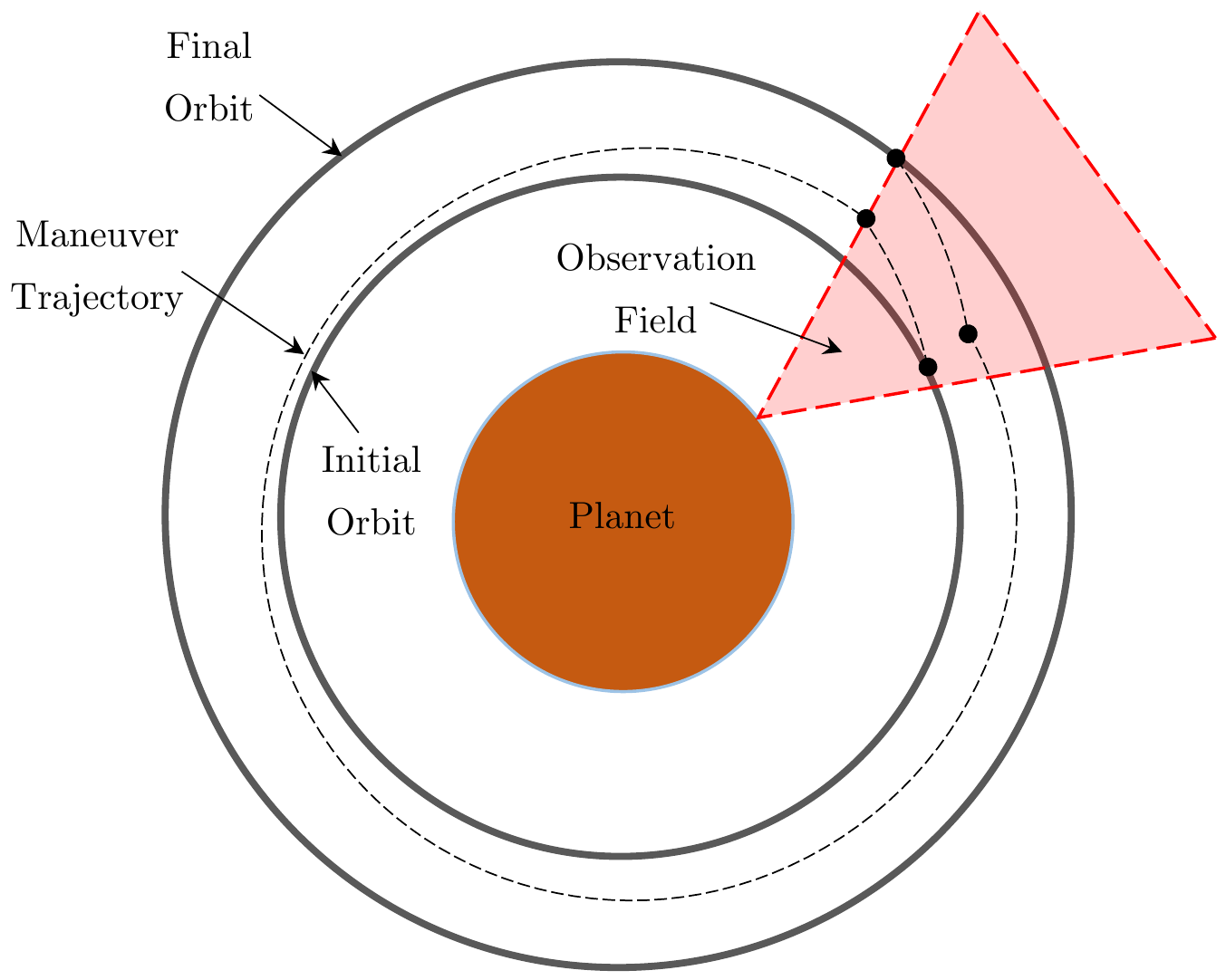}
\caption{Schematic geometric view of the problem discussed in this paper.}
\label{fig:0p}
\end{figure}

The remaining parts of this paper are organized as follows: In Section \ref{S:2}, the system is modeled, the Lambert's algorithm is inserted to the formulations, and the actuated system is introduced. Section \ref{S:3} is devoted to the covariance analysis and propagation in a simplified case and some results are presented. The optimization and simulations are performed in Section \ref{S:4}. Finally, concluding remarks and future directives are presented in Section \ref{S:5}.

\section{Orbital Dynamics and Impulsive Control}
\label{S:2}

The spacecraft dynamics is assumed to be the unperturbed two-body problem. Let $\textbf{r},\textbf{v},\textbf{a}\in\mathbb{R}^3$ denote the position, velocity, and acceleration of the spacecraft in an Earth-centered inertial coordinate system, respectively, and $t\in[0,\infty)$ denote the time. Therefore:
\begin{equation}
\label{eq:1}
\textbf{a}(t)=-\frac{\mu}{r(t)^3}\textbf{r}(t)
\end{equation}
in which $r=\|\textbf{r}\|$ stands for the Euclidean norm. The following discrete-time form is obtainable using Eq. \eqref{eq:1}: 
\begin{equation}
\label{eq:2}
\textbf{r}_{i+1}=\textbf{f}_r\left(\textbf{r}_i,\textbf{v}_i,\delta t\right)
\end{equation}
\begin{equation}
\label{eq:3}
\textbf{v}_{i+1}=\textbf{f}_v\left(\textbf{r}_i,\textbf{v}_i,\delta t\right)
\end{equation}
where $\textbf{r}_i,\textbf{v}_i\in\mathbb{R}^3$ are the position and velocity vectors of the spacecraft at step $i$ which occurs at $t_i=t_{i-1}+\delta t_{i-1,i}\in[0,\infty)$. Functions $\textbf{f}_r(\cdot,\cdot,\cdot),\textbf{f}_v(\cdot,\cdot,\cdot):\mathbb{R}^3\times\mathbb{R}^3\times[0,\infty)\mapsto\mathbb{R}^3$ are defined as below: 
\begin{equation}
\label{eq:4}
\textbf{f}_r(\textbf{r}_i,\textbf{v}_i,t)=\int_t\textbf{f}_v(\textbf{r}_i,\textbf{v}_i,\tau)d\tau+\textbf{r}_i
\end{equation}
\begin{equation}
\label{eq:5}
\textbf{f}_v(\textbf{r}_i,\textbf{v}_i,t)=\int_t\textbf{a}(\tau)d\tau+\textbf{v}_i
\end{equation}

Let us introduce two symbols of $\textbf{v}_i^-,\textbf{v}_i^+\in\mathbb{R}^3$ in order to denote the spacecraft velocity vector before and after applying the impulse vector, respectively. So,
\begin{equation}
\label{eq:6}
\textbf{v}_i^+=\textbf{v}_i^-+\delta\textbf{v}_i
\end{equation}
in which $\delta\textbf{v}_i\in\mathbb{R}^3$ is the impulse vector at step $i$. Using the above notation, the actuated analog of Eqs. \eqref{eq:2} and \eqref{eq:3} can be written as: 
\begin{equation}
\label{eq:7}
\textbf{r}_{i+1}=\textbf{f}_r\left(\textbf{r}_i,\textbf{v}_i^+,\delta t_{i,i+1}\right)\equiv \textbf{f}_r\left(\textbf{r}_i,\textbf{v}_i^-+\delta\textbf{v}_i,\delta t_{i,i+1}\right)
\end{equation}
\begin{equation}
\label{eq:8}
\textbf{v}_{i+1}^-=\textbf{f}_v\left(\textbf{r}_i,\textbf{v}_i^+,\delta t_{i,i+1}\right)\equiv\textbf{f}_v\left(\textbf{r}_i,\textbf{v}_i^-+\delta\textbf{v}_i,\delta t_{i,i+1}\right)
\end{equation}

It is presumed that both the initial and final orbits rotate clockwise (CW) or counterclockwise (CCW) which is the case for almost all IOMs. Knowing the values of $\textbf{r}_i$, $\textbf{r}_{i+1}$, and $\delta t_{i,i+1}$, the velocity vectors of $\textbf{v}_i^+$ and $\textbf{v}_{i+1}^-$ can be obtained by implementation of Lambert's algorithm (Algorithm 54 in \citep{vallado}). First, Consider the following assumption. 

\begin{assumption}
\label{ass:1}
The angle between two subsequent impulse positions is less than $\pi$, i.e., $\forall i\in\mathbb{N}:|\angle(\textbf{r}_i,\textbf{r}_{i+1})|<\pi$. In other words, dividing the $\mathbb{R}^2$ space into two subsets $\mathcal{R}_1=\{\textbf{r}\in\mathbb{R}^2:|\angle(\textbf{r},\textbf{r}_{i+1})|+|\angle(\textbf{r},\textbf{r}_{i})|=|\angle(\textbf{r}_i,\textbf{r}_{i+1})|\}$ and $\mathcal{R}_2=\mathbb{R}^2-\mathcal{R}_1$, if for a two-body dynamics with the initial conditions $\textbf{r}_i$ and $\textbf{v}_{i}^-$ the trajectory enters $\mathcal{R}_1$, then it is said that this assumption is satisfied. Otherwise, this assumption is not satisfied. See Fig. \ref{fig:ass1} 
\end{assumption}

\begin{figure}[h]
\centering\includegraphics[width=0.5\linewidth]{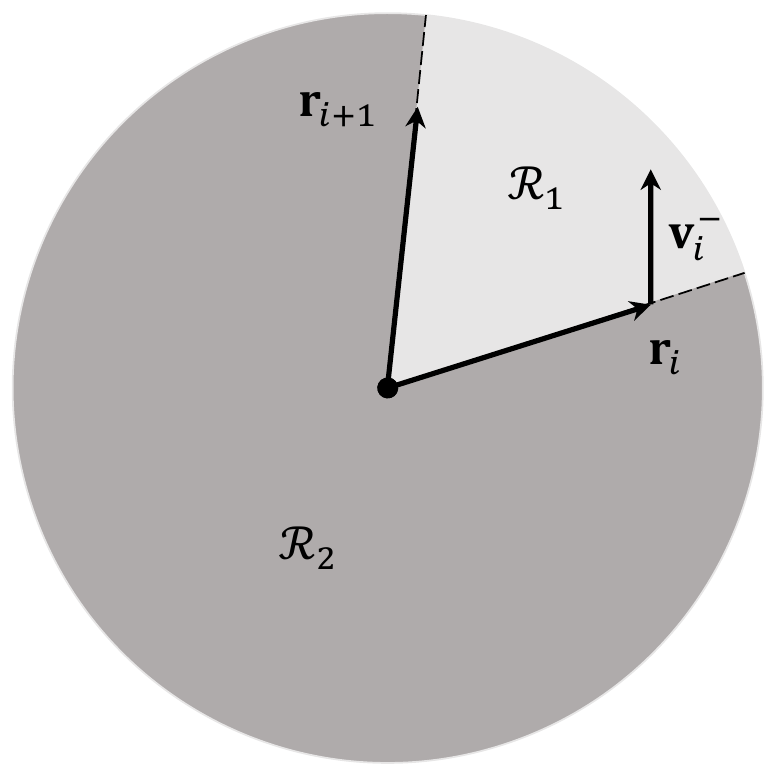}
\caption{Schematic figure for Assumption \ref{ass:1}.}
\label{fig:ass1}
\end{figure}

The Lambert's algorithm can give the velocity vectors of a trajectory that goes through $\textbf{r}_i$ and $\textbf{r}_{i+1}$ with a time interval of $\delta t_{i,i+1}$. The output of the Lambert's algorithm is unique under Assumption \ref{ass:1} (for more details see \citep{simo1973solucion}). Let $\textbf{L}(\cdot,\cdot,\cdot):\mathcal{D}_L\mapsto\mathbb{R}^3$ be a function that employs the Lambert's algorithm where $\mathcal{D}_L\subset\mathbb{R}^3\times\mathbb{R}^3\times[0,\infty)$ is considered such that Assumption \ref{ass:1} holds:
\begin{equation}
\label{eq:9}
\textbf{v}_i^+=\textbf{L}\left(\textbf{r}_i,\textbf{r}_{i+1},\delta t_{i,i+1}\right)
\end{equation}

Therefore, using Eqs. \eqref{eq:9} and \eqref{eq:6} for a two-impulse maneuver from $\textbf{r}_i,\textbf{v}_i^-$ to $\textbf{r}_{i+1},\textbf{v}_{i+1}^+$ in a time interval of $\delta t_{i,i+1}$, the first impulse vector can be obtained as follows: 
\begin{equation}
\label{eq:10}
\delta\textbf{v}_i=\textbf{L}\left(\textbf{r}_i,\textbf{r}_{i+1},\delta t_{i,i+1}\right)-\textbf{v}_i^-
\end{equation}
and using Eqs. \eqref{eq:8}, \eqref{eq:9}, and \eqref{eq:6} the second impulse vector is:
\begin{equation}
\label{eq:11}
\delta\textbf{v}_{i+1}=\textbf{v}_{i+1}^+-\textbf{f}_v\left(\textbf{r}_i,\textbf{L}\left(\textbf{r}_i,\textbf{r}_{i+1},\delta t_{i,i+1}\right),\delta t_{i,i+1}\right)
\end{equation}

On the other hand, to handle those cases where Assumption \ref{ass:1} is not satisfied (i.e., $|\angle(\textbf{r}_i,\textbf{r}_{i+1})|\in(\pi,2\pi)$), first we need to consider the following proposition:

\begin{proposition}
\label{prop:1}
Let $\textbf{r}_{ac}$ denote the position where orbits (a) and (c) intersect and similarly, $\textbf{r}_{bc}$ denote the position where orbits (a) and (b) intersect which are shown in Fig. \ref{fig:001p}. Consider a spacecraft decides to travel from orbit (a) to (b) using an arc of orbit (c). Then, the impulse magnitudes are equal in the following scenarios: 
\begin{enumerate}
\item Orbits (a) and (b) are CCW and for a two-impulse maneuver between $\textbf{r}_{ac}$ and $\textbf{r}_{bc}$, an arc of orbit (c) is used which has a CCW rotation (i.e., orbit (c)--Traj. (1) in Fig. \ref{fig:001p}).
\item Orbits (a) and (b) are CW and for a two-impulse maneuver between $\textbf{r}_{ac}$ and $\textbf{r}_{bc}$, an arc of orbit (c) is used which has a CW rotation (i.e., orbit (c)--Traj. (2) in Fig. \ref{fig:001p}).
\end{enumerate}
\end{proposition}

\begin{proof}
Suppose the velocity of orbit (a) in $\textbf{r}_{ac}$ is shown by $\textbf{v}_{ac(a)}$ or $-\textbf{v}_{ac(a)}$ when the rotation is CW or CCW, respectively. This rule is then used to denote the rest of the velocities. In scenario (2) the impulse magnitudes are $\|\delta\textbf{v}_{1(CW)}\|=\|\textbf{v}_{ac(c)}-\textbf{v}_{ac(a)}\|$ and $\|\delta\textbf{v}_{2(CW)}\|=\|\textbf{v}_{bc(b)}-\textbf{v}_{bc(c)}\|$. In scenario (1) the impulse magnitudes are $\|\delta\textbf{v}_{1(CCW)}\|=\|-\textbf{v}_{ac(c)}+\textbf{v}_{ac(a)}\|$ and $\|\delta\textbf{v}_{2(CCW)}\|=\|-\textbf{v}_{bc(b)}+\textbf{v}_{bc(c)}\|$. Therefore, $\|\delta\textbf{v}_{1(CW)}\|=\|\delta\textbf{v}_{1(CCW)}\|$ and $\|\delta\textbf{v}_{2(CW)}\|=\|\delta\textbf{v}_{2(CCW)}\|$, and the statement is proved. 
\end{proof}

\begin{figure}[h]
\centering\includegraphics[width=1\linewidth]{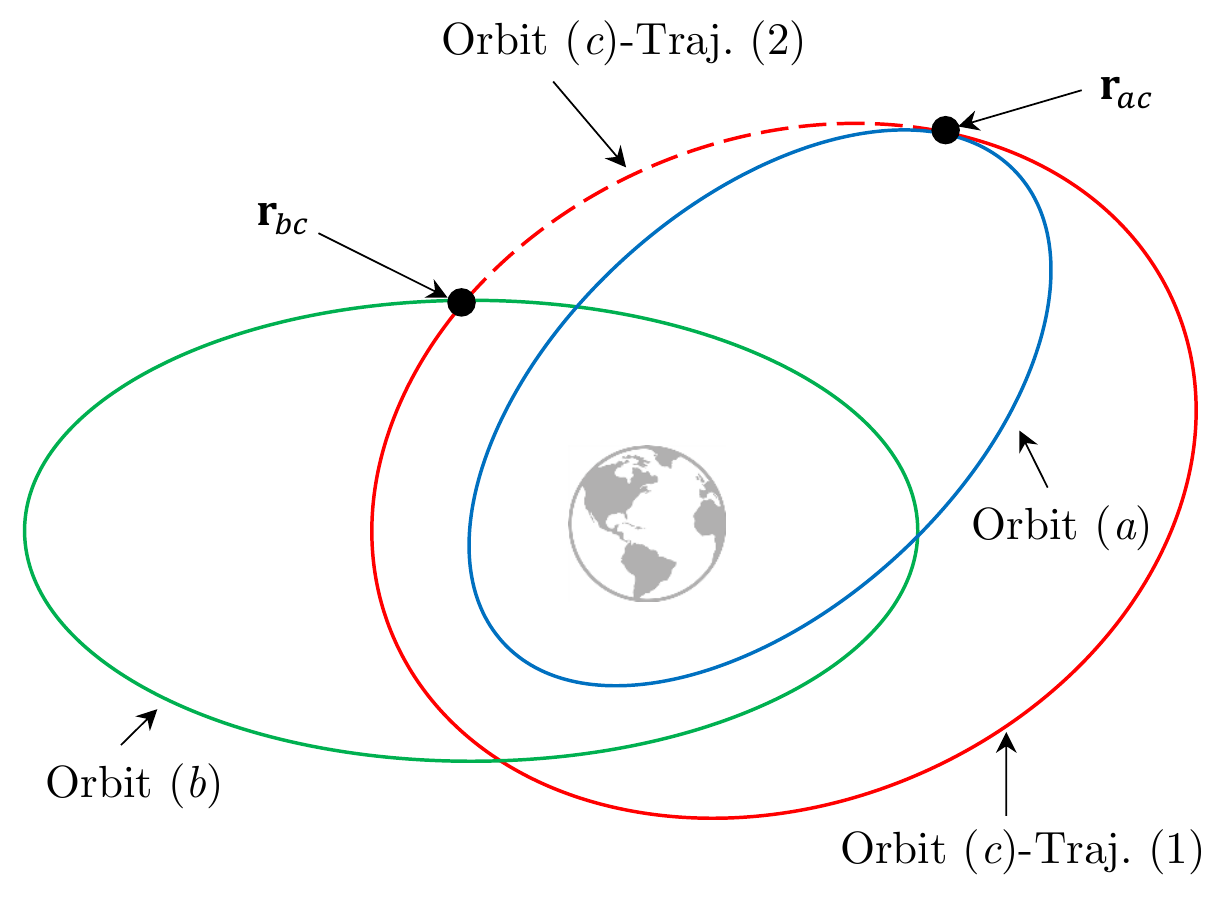}
\caption{Visualization of the parameters used in Proposition \ref{prop:1}.}
\label{fig:001p}
\end{figure}

According to Proposition \ref{prop:1}, if Assumption \ref{ass:1} is not satisfied, then the problem can be viewed as an equivalent problem at which Assumption \ref{ass:1} is satisfied. The procedure for the calculation of an impulse vector is summarized in Algorithm \ref{al:1} and depicted in Fig. \ref{fig:001}. Algorithm \ref{al:2} presents the extended procedure for the case of MIOM. 

\begin{algorithm}
\label{al:1}
\caption{An algorithm based on Lambert's problem for impulse generation.}
\SetAlgoLined
\textbf{Input:} Initial position $\textbf{r}_i$; final position $\textbf{r}_{i+1}$; transfer time $\delta t_{i,i+1}$; and initial velocity $\textbf{v}^-_i$.\\
\textbf{Output:} Impulse vector $\delta \textbf{v}_i$; and the final velocity $\textbf{v}_{i+1}^-$.\\
\If {Assumption \ref{ass:1} is satisfied}{
1. Use Eq. \eqref{eq:10} and calculate $\delta \textbf{v}_i$.\\
2. Use Eq. \eqref{eq:8} and calculate $\textbf{v}_{i+1}^-$. \\
}
\If {Assumption \ref{ass:1} is not satisfied}{
1. $\textbf{v}^-_i\gets-\textbf{v}^-_i$\\
2. Use Eq. \eqref{eq:10} and calculate $\delta \textbf{v}_i$.\\
3. Use Eq. \eqref{eq:8} and calculate $\textbf{v}_{i+1}^-$. \\
4. $\delta \textbf{v}_i\gets-\delta \textbf{v}_i$ \\
5. $\textbf{v}_{i+1}^-\gets-\textbf{v}_{i+1}^-$
}
\textbf{Return:} $\delta \textbf{v}_i$, $\textbf{v}_{i+1}^-$.
\end{algorithm}

\begin{figure}[h]
\centering\includegraphics[width=0.7\linewidth]{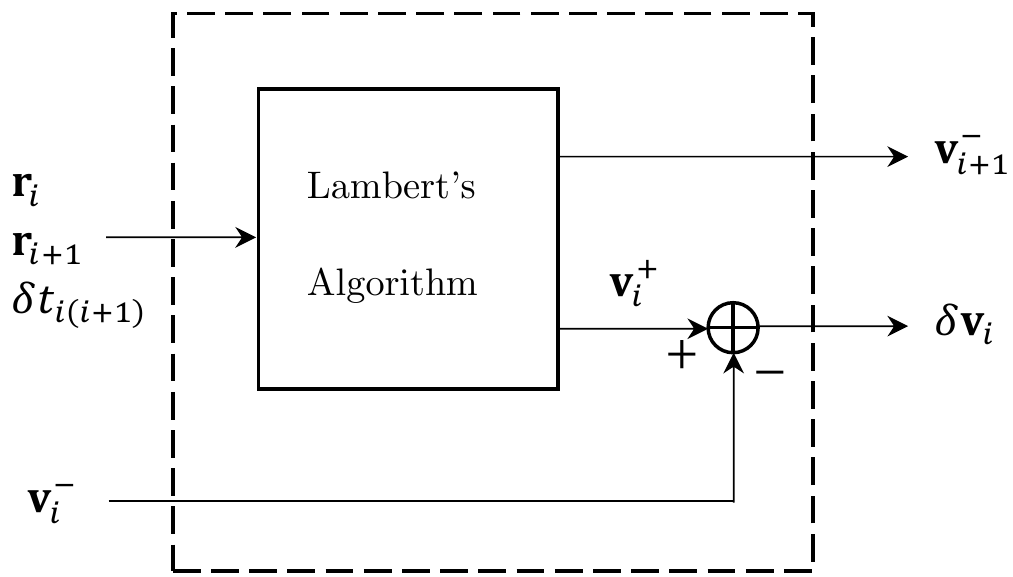}
\caption{Schematic view for Algorithm \ref{al:1}.}
\label{fig:001}
\end{figure}

\begin{algorithm}
\label{al:2}
\caption{An algorithm for impulse generation in MIOMs with $n$ impulses.}
\SetAlgoLined
\textbf{Input:} Number/index of impulses, $i=1,\cdots,n$; impulse positions $\textbf{r}_i$; impulse times, $\delta t_{i,i+1}$; initial velocity $\textbf{v}^-_1$; and final velocity, $\textbf{v}^+_n$.\\
\textbf{Output:} Impulse vectors $\delta \textbf{v}_i$, $i=1,\cdots,n$.\\
\For {$i=1,\cdots,n$}{
1. Run Algorithm \ref{al:1} with $\textbf{r}_i$, $\delta t_{i,i+1}$, and $\textbf{v}^-_i$ as inputs.\\
2. Save the first output of Algorithm \ref{al:1}, $\delta \textbf{v}_i$, and use the second output, $\textbf{v}_{i+1}^-$, as an input for the next iteration. \\
}
3. $\delta\textbf{v}_n\gets\textbf{v}_n^+-\textbf{v}_n^-$. \\
\textbf{Return:} $\delta \textbf{v}_i$, $i=1,\cdots,n$.
\end{algorithm}

Consider the following cost functions for a two impulse maneuver:
\begin{equation}
\label{eq:11p}
J_{CE}=\sum_{i=1}^n\delta v_i
\end{equation}
\begin{equation}
\label{eq:11pp}
J_{MI}=\max\{\delta v_1,\delta v_2,\cdots,\delta v_n\}
\end{equation}
in which the former defines the control effort, the latter is the maximum required impulse magnitude, and $\|\delta\textbf{v}_i\|=\delta v_i$. The above-mentioned cost functions are expressible as follows using Eqs. \eqref{eq:9}, \eqref{eq:10}, and \eqref{eq:11}. 

\begin{equation}
\label{eq:11ppp}
\begin{split}
&J_{CE}=\left\|\textbf{L}\left(\textbf{r}_1,\textbf{r}_2,\delta t_{1,2}\right)-\textbf{v}_1^-\right\|\\
&+\left\|\textbf{L}\left(\textbf{r}_2,\textbf{r}_{3},\delta t_{2,3}\right)-\textbf{f}_v\left(\textbf{r}_1,\textbf{L}\left(\textbf{r}_1,\textbf{r}_2,\delta t_{1,2}\right),\delta t_{1,2}\right)\right\|+\cdots
\end{split}
\end{equation}
\begin{equation}
\label{eq:11pppp}
\begin{split}
&J_{MI}=\max\left\{\left\|\textbf{L}\left(\textbf{r}_1,\textbf{r}_2,\delta t_{1,2}\right)-\textbf{v}_1^-\right\|,\right.\\
&\left.\left\|\textbf{L}\left(\textbf{r}_2,\textbf{r}_{3},\delta t_{2,3}\right)-\textbf{f}_v\left(\textbf{r}_1,\textbf{L}\left(\textbf{r}_1,\textbf{r}_2,\delta t_{1,2}\right),\delta t_{1,2}\right)\right\|,\cdots\right\}
\end{split}
\end{equation}

To simplify the analysis for near circular orbits, polar coordinates are used. In a polar coordinate system, an equatorial orbit can be specified by $r(t)$ and $\theta(t)$ such that $\textbf{r}(t)=r(t)[\cos\theta(t)\\\quad\sin\theta(t)]^T$. Before presenting a result, consider the following assumption:
\begin{assumption}
\label{ass:2}
The spacecraft trajectory, except of impulse instants, is approximately circular, i.e., $\dot{r}(t)\simeq0$ and $\ddot{\theta}(t)\simeq0$.
\end{assumption}

Under Assumption \ref{ass:2}, in a coplanar trajectory, $\textbf{v}_i^-,\textbf{r}_i$ and $\textbf{v}_{i+1}^+,\textbf{r}_{i+1}$ are functions of $\theta_i$ and $\theta_{i+1}$, respectively. Therefore, the cost functions of Eqs. \eqref{eq:11ppp} and \eqref{eq:11pppp} are both functions of $\theta_i$, $\theta_{i+1}$, and $\delta t_{i,i+1}$, i.e., $J_{CE}\equiv J_{CE}(\theta_1,\cdots,\theta_n,\delta t_{1,2},\cdots,\delta t_{n-1,n})$ and $J_{MI}\equiv J_{MI}(\theta_1,\cdots,\theta_n,\delta t_{1,2},\cdots,\delta t_{n-1,n})$. 

\begin{proposition}
\label{fact:1}
Suppose Assumption \ref{ass:1} is satisfied. The cost functions $J_{CE}$ and $J_{MI}$ both have one global minimum value with respect to $\delta t_{i,i+1}$ and if also Assumption \ref{ass:2} is satisfied, then, they have also one global minimum with respect to $\theta_i$ for all $i\in\mathbb{N}$ (See Fig. \ref{fig:0}).
\end{proposition}

\begin{proof}
For simplicity we drop the subscripts. The semimajor axis of the transfer trajectory, $a$, against the time of flight, $\delta t$, has a single global minimum value (which corresponds to the so-called minimum-energy transfer \citep{izzo2015revisiting}). According to the orbital energy equation, $v^2/2=\mu(r-1/a)$ ($v$ can stand for the initial or final velocity of the Lambert's trajectory), $v$ is a non-decreasing function of $a$, therefore $v$ has a single global minimum with respect to $\delta t$. The square of the (first or second) impulse magnitude, $\delta v^2$, has a quadratic relation with $v$. Therefore, since $\delta v^2$ is a composite of a quadratic function and an invex function of $\delta t$, it has none, a single, or two extremums with respect to $\delta t$ (because $d\delta v/d\delta t=d\delta v/da\cdot da/d\delta t$, hence, $d\delta v/d\delta t$ can switch the sign one, two, or three times). We know that both when $t\rightarrow0$ and $t\rightarrow\infty$ the value of $\delta v^2$ approaches infinity. Therefore, $\delta v$ (as well as their weighted sum) has a single global minimum against $\delta t$ (which according to Fig. \ref{fig:0} it is not convex since obviously counter examples exist). 

For the second part we use a proof by contradiction. Suppose that the cost function with respect to $\theta$, for a constant $\delta t$, changes the derivative sign of the cost function (and equivalently the velocity magnitudes) three (or more) times. Therefore, corresponding to some values of $\delta t$, $a$ (or $v$), and $J$, there exist three (or more) solution values of $\theta$. According to Eq. (63) in \citep{de2018solution} one can obtain:
\begin{equation}
\label{eq:prop_1}
\delta t^2=\alpha_2\cos^2\left(\frac{\theta}{2}\right)+\alpha_1\cos\left(\frac{\theta}{2}\right)+\alpha_0
\end{equation}
for some fixed values of $\alpha_{0,1,2}$. According to Eq. \eqref{eq:prop_1} it is impossible for $\theta$ to pick more than two solutions. Therefore, the first assumption is false and the cost function with respect to $\theta$, for a constant $\delta t$, has none, one, or two solutions which means that there exists one global minimum. 
\end{proof}

\begin{figure*}[h]
\centering\includegraphics[width=1\linewidth]{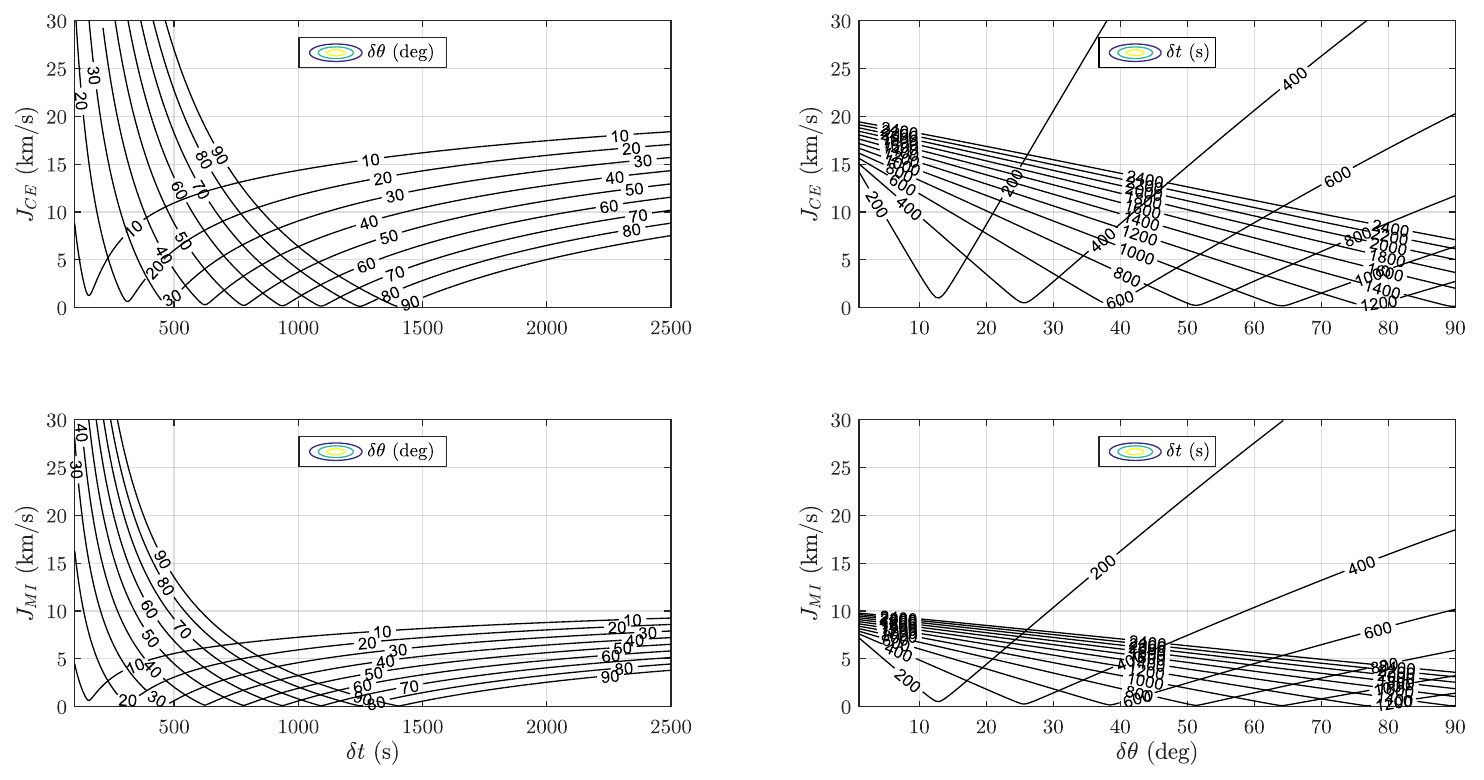}
\caption{Graphical representation of $J_{EC}$ and $J_{MI}$ as functions of $\delta t_{1,2}$ and $\delta\theta=\theta_2-\theta_1$ for a two-impulse maneuver between circular orbits with altitudes of 400 km and 500 km.}
\label{fig:0}
\end{figure*}

\section{Covariance Propagation and Analysis}
\label{S:3}

For a circular orbit, the dynamics in a polar coordinate system is expressible as follows: 

\begin{equation}
\label{eq:12}
\left[\begin{matrix}
\dot{r}(t) \\
\dot{\theta}(t) \\
\dot{\omega}(t)\end{matrix}\right]=\left[\begin{matrix}
0 & 0 & 0 \\
0 & 0 & 1 \\
0 & 0 & 0
\end{matrix}\right]\left[\begin{matrix}
r(t) \\
\theta(t) \\
\omega(t)\end{matrix}\right]+\textbf{w}
\end{equation}
where $\textbf{w}\in\mathbb{R}^3$ is a zero-mean, normally-distributed, random vector with the associated diagonal covariance matrix of $\textbf{Q}=[Q_{ii}]\in\mathbb{R}^{3\times3}$. Suppose the measurement vector is $\textbf{m}=[r\quad\theta]^T+\textbf{v}$ in which $\textbf{v}$ is a zero-mean, normally-distributed random vector with the associated diagonal covariance matrix of $\textbf{R}=[R_{ii}]\in\mathbb{R}^{2\times2}$.

The propagation of the state covariance matrix, $\textbf{P}(t)=[P_{ij}(t)]$, can be stated as

\begin{equation}
\label{eq:13}
\dot{\textbf{P}}=\left[\begin{matrix}
0 & 0 & 0 \\
0 & 0 & P_{22} \\
0 & P_{22} & 2P_{23}
\end{matrix}\right]-\left[\begin{matrix}
P_{11}^2/R_{11} & 0 & 0 \\
0 & P_{22}^2/R_{22} & P_{22}P_{23}/R_{22} \\
0 & P_{22}P_{23}/R_{22} & P_{33}^2/R_{22}
\end{matrix}\right]+\textbf{Q}
\end{equation}

According to the above formulation, $P_{11}$ and $P_{22}$ are decoupled. Thus, for $i=1$ or $2$:

\begin{equation}
\label{eq:14}
P_{ii}(t)=\frac{R_{ii}P_{ii}(0)+\sqrt{R_{ii}Q_{ii}}(P_{ii}(0)-\sqrt{R_{ii}Q_{ii}})t}{R_{ii}+(P_{ii}(0)-\sqrt{R_{ii}Q_{ii}})t}
\end{equation}

In Fig. \ref{fig:1}, the phase plane of Eq. \eqref{eq:12} is plotted schematically which shows that if $P_{ii}(0)>\sqrt{R_{ii}Q_{ii}}$, that is often the case, the uncertainty of $r(t)$ and $\theta(t)$ will decrease over time. If $Q_{ii}=0$, then the limit of uncertainty is zero. 

\begin{figure}[h]
\centering\includegraphics[width=0.59\linewidth]{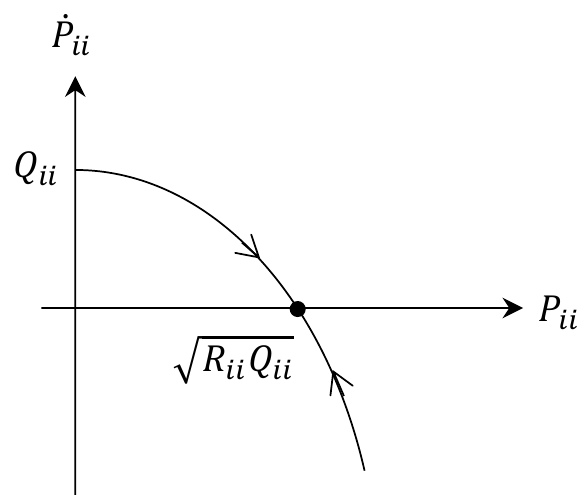}
\caption{Phase plane of $\dot{P}_{ii}=-P_{ii}^2/R_{ii}+Q_{ii}$.}
\label{fig:1}
\end{figure}

Considering $\theta(t)=\omega t=\sqrt{\mu/r^3} t$ for circular orbits, Eq. \eqref{eq:14} can be expressed as a function of $\theta$:

\begin{equation}
\label{eq:15}
P_{ii}(\theta)=\frac{\sqrt{\mu}R_{ii}P_{ii}(0)+r^{3/2}\sqrt{R_{ii}Q_{ii}}(P_{ii}(0)-\sqrt{R_{ii}Q_{ii}})\theta}{\sqrt{\mu}R_{ii}+r^{3/2}(P_{ii}(0)-\sqrt{R_{ii}Q_{ii}})\theta}
\end{equation}

According to Eqs. \eqref{eq:14} and \eqref{eq:15}, the variances of $r(t)$ and $\theta(t)$ decrease in the window that observation exists, but in the rest of the path, the variances increase as $P_{ii}=Q_{ii}t=\sqrt{r^3/\mu}Q_{ii}\theta$. Depending on the measurement accuracy, spacecraft orbit, and the process uncertainty, waiting for the next window may increase or decrease the orbit accuracy. However, the variance decrement in the observation window still exists. Therefore, another cost function on behalf of the variances can be defined as follows:

\begin{equation}
\label{eq:16}
J_V=\sum_{i=1}^{n-1}|\theta_i-\theta_i^{\mathrm{max}}|
\end{equation}

Another source of uncertainty is the impulse vectors. A lower amount of $J_{CE}$ results in a more accurate final trajectory. In this regard, inserting $J_{CE}$ to the total cost function not only results in a decrease of control effort, but as a side effect, decreases the uncertainty due to the impulse magnitudes. 

\section{Optimization and Simulations}
\label{S:4}

According to the results of the previous sections, the following total cost function can be defined which under Assumptions \ref{ass:1} and \ref{ass:2} is (quasi) convex with respect to the optimization variables of $\theta_i$, $i=\{1,2,\cdots,n\}$ and $\delta t_{i,i+1}$, $i=\{1,2,\cdots,n-1\}$. 

\begin{equation}
\label{eq:17}
\begin{split}
&\underset{\theta_i,\delta t_{i,i+1}}{\mathrm{minimize}}\quad w_{CE}J_{CE}+w_{MI}J_{MI}+w_VJ_V\\
&\quad\quad\quad\quad\quad\theta_i^{\mathrm{min}}\leq\theta_i\leq\theta_i^{\mathrm{max}}
\end{split}
\end{equation}

The above problem can be easily solved under convex constraints on $\delta t_{i,i+1}$ and $\theta_i$. The minimum and maximum values of $\theta_i$ are constrained by $\theta_i^{\mathrm{min}}$ and $\theta_i^{\mathrm{max}}$ that are determined by the observation field provided by a ground based observation site for example. Since a real MIOM problem has degrees of freedom also on the values of $\|\textbf{r}_i\|$, therefore, the following problem is solved in this paper instead of the ideal problem defined by Eq. \eqref{eq:17}. The following problem may has more than one local minimum solutions, but according to Proposition \ref{fact:1}, the local minimum is near the global minimum if the initial and final orbits are close enough and consequently the transfer trajectory is near circular (i.e., Assumption \ref{ass:2} is satisfied approximately).
\begin{equation}
\label{eq:17_p}
\begin{split}
&\underset{\theta_i,\delta t_{i,i+1},\|\textbf{r}_i\|}{\mathrm{minimize}}\quad w_{CE}J_{CE}+w_{MI}J_{MI}+w_VJ_V\\
&\quad\quad\quad\quad\quad\theta_i^{\mathrm{min}}\leq\theta_i\leq\theta_i^{\mathrm{max}}
\end{split}
\end{equation}

\begin{remark}
\label{rem:1}
The cost function which is considered in this paper is free of exact covariance elements, unlike works done by \cite{in20,in21,in22}, in which a representative cost function, $J_V$, is used instead. This approach makes the solution much more easier and faster to obtain, while needs the designer to have intuitions about the level of the uncertainties in order to select appropriate values for the weighting parameter, $w_V$.
\end{remark}

\begin{remark}
\label{rem:2}
An optimization problem pretty similar to problem \eqref{eq:17_p} may be written in the following form:
\begin{equation}
\label{eq:17_pp}
\begin{split}
&\underset{\theta_i,\delta t_{i,i+1},\|\textbf{r}_i\|}{\mathrm{minimize}}\quad w_{CE}J_{CE}+w_{MI}J_{MI}\\
&\quad\quad\quad\quad\quad\theta_i^{\mathrm{min}}+\theta_i^{\mathrm{lower}}\leq\theta_i\leq\theta_i^{\mathrm{max}}
\end{split}
\end{equation}
such that after $\theta_i=\theta_i^{\mathrm{min}}+\theta_i^{\mathrm{lower}}$ it is known that a convergence occurs in the filtering procedure. However, since such an information is not known for a system (i.e., the covariance matrix elements cannot be calculated offline) we solve the previously mentioned problem \eqref{eq:17_p} in this study. 
\end{remark}

In this paper, a gradient-based optimization method is used at which the gradients are evaluated numerically by a finite-difference technique. The implemented optimization method is summarized in Algorithm \ref{al:3}. 

\begin{algorithm}
\label{al:3}
\caption{A gradient-based numerical optimization algorithm to solve problem \eqref{eq:17_p}.}
\SetAlgoLined
\textbf{Input:} A desired number of impulses, $i=1,\cdots,n$; initial guesses for impulse position magnitudes, $\|\textbf{r}_i^{(1)}\|$ ($i=2,\cdots,n-1$); initial guesses for impulse angles, $\theta_i^{(1)}$ ($i=1,\cdots,n-1$); initial guesses for impulse times, $\delta t_{i,i+1}^{(1)}$ ($i=1,\cdots,n-1$); initial orbit, $\|\textbf{r}_1\|$ and $\|\textbf{v}^-_1\|$ as well as the final orbit, $\textbf{r}_n$ and $\textbf{v}^+_n$; the observation filed of view, $\alpha$; and the weighting values, $w_{CE}$, $w_{MI}$, and $w_{V}$.\\
\textbf{Output:} The optimum values for optimization variables, $\|\textbf{r}_i^{*}\|$, $\theta_i^{*}$, and $\delta t_{i,i+1}^{*}$.\\
1. $\textbf{x}^{(1)}\gets[\|\textbf{r}_2^{(1)}\|,\cdots,\|\textbf{r}_{n-1}^{(1)}\|,\theta_1^{(1)},\cdots,\theta_{(n-1)}^{(1)},\delta t_{1,2}^{(1)},\cdots,\delta t_{n-1,n}^{(1)}]^T$ \\
\For {$i=1,\cdots,N$}{
2. $\theta_n^{(i)}\gets\theta^{\mathrm{max}}_n$\\
3. Pick an appropriate value for $\gamma>0$. \\
\For {$j=1,\cdots,4n-7$}{
4. Run Algorithm \ref{al:2} with $\textbf{x}^{(i)}$, $\textbf{r}_1$, $\textbf{v}^-_1$, $\textbf{r}_n$, and $\textbf{v}^+_n$. \\
5. $J\gets w_{CE}J_{CE}+w_{MI}J_{MI}+w_VJ_V$ (from the outputs of step 4) \\
6. Pick a small enough $\varepsilon>0$.\\
7. Add $\varepsilon$ to the $j$th element of $\textbf{x}^{(i)}$. \\
8. Repeat step 4. \\
9. $J_{\delta}\gets w_{CE}J_{CE}+w_{MI}J_{MI}+w_VJ_V$ (from the outputs of step 8) \\
10. $(\nabla_{\textbf{x}} J)_j\gets J_{\delta}-J$
}
11. $(\nabla_{\textbf{x}} J)_{\textbf{x}=\textbf{x}^{(i)}}\gets[(\nabla_{\textbf{x}} J)_1,\cdots,(\nabla_{\textbf{x}} J)_{4n-7}]^T$ \\
12. $\textbf{x}^{(i+1)}\gets\textbf{x}^{(i)}-\gamma(\nabla_{\textbf{x}} J)_{\textbf{x}=\textbf{x}^{(i)}}$ \\
13. Extract the output values from $\textbf{x}^{(i)}=[\|\textbf{r}_2^{(i)}\|,\cdots,\|\textbf{r}_{n-1}^{(i)}\|,\theta_1^{(i)},\cdots,\theta_{(n-1)}^{(i)},\delta t_{1,2}^{(i)},\cdots,\delta t_{n-1,n}^{(i)}]^T$. \\
\For {$i=1$ to $n-1$}{
14. Obtain the values of $\theta_i^{\mathrm{max}}$ and $\theta_i^{\mathrm{min}}$ (e.g. using Eq. \eqref{eq:18}).\\
\If {$\theta_i>\theta_i^{\mathrm{max}}$}{
15. $\theta_i\gets\theta_i^{\mathrm{max}}$
}
\If {$\theta_i<\theta_i^{\mathrm{min}}$}{
16. $\theta_i\gets\theta_i^{\mathrm{min}}$
}
}
}
17. Extract the output values from $\textbf{x}^{(N)}\equiv\textbf{x}^{*}=[\|\textbf{r}_2^{*}\|,\cdots,\|\textbf{r}_{n-1}^{*}\|,\theta_1^{*},\cdots,\theta_{(n-1)}^{*},\delta t_{1,2}^{*},\cdots,\delta t_{n-1,n}^{*}]^T$. \\
\textbf{Return:} $\|\textbf{r}_i^{*}\|$ ($i=2,\cdots,n-1$), $\theta_i^*$ ($i=1,\cdots,n-1$), $\theta_n^*=\theta_n^{\mathrm{max}}$, and $\delta t_{i,i+1}^{*}$ ($i=1,\cdots,n-1$).
\end{algorithm}

If the observation field of view half angle is $\alpha$, then we have:

\begin{equation}
\label{eq:18}
\theta_i^{\mathrm{max}}=\cos^{-1}\left(\frac{R}{R+r_i}\sin^2\alpha+\cos\alpha\sqrt{1-\left(\frac{R}{R+r_i}\right)^2\sin^2\alpha}\right)
\end{equation}
where $R$ is the celestial body radius and $\theta_i^{\mathrm{min}}$ is considered equal to $-\theta_i^{\mathrm{max}}$ in our case studies. 

\begin{remark}
\label{rem:3}
It is obvious that the set of all $n$-impulse maneuvers is a subset of the set of all $m$-impulse maneuvers if $m>n$ (which is equivalent to the set of $m$-impulse maneuvers with $m-n$ zero impulses). Consequently, we have $J^*_m\leq J^*_n$ where $J^*_m$ and $J^*_n$ are the solutions of problem \eqref{eq:17_p} corresponding to $m$ and $n$ impulses, respectively. Therefore, increasing the number of impulses may decrease the optimal cost function or at least leaves it unchanged. However, the computational effort will increase considerably which is a result of the curse of dimensionality. 
\end{remark}

The four impulse orbital maneuver is considered as a numerical example for a transfer between circular orbits with altitudes of $500$ km and $1000$ km with $\alpha=60^{\circ}$. The unconstrained optimal maneuver for this case in view of control effort is the Hohmann maneuver that requires a maximum impulse capability of $0.33$ km/s and the sum of impulses will be as low as $0.62$ km/s. 

Fig. \ref{fig:2} shows the optimal trajectory for $w_{CE}=1,w_{MI}=w_{V}=0$ at which the sum of impulses is $1.48$ km/s and requires a maximum impulse magnitude of $0.78$ km/s. This solution has the minimum $J_{CE}$ such that the impulse position angles are constrained to be located in the observation window. However, this trajectory is not appropriate in view of uncertainty since $J_{V}$ is not considered, and consequently, the  first impulse is applied at the beginning of the observation period when the filtering procedure have not had enough time to converge. The sum of impulses is much higher than the Hohmann solution which is the result of the constrained impulse positions. 

\begin{figure}[h]
\centering\includegraphics[width=1\linewidth]{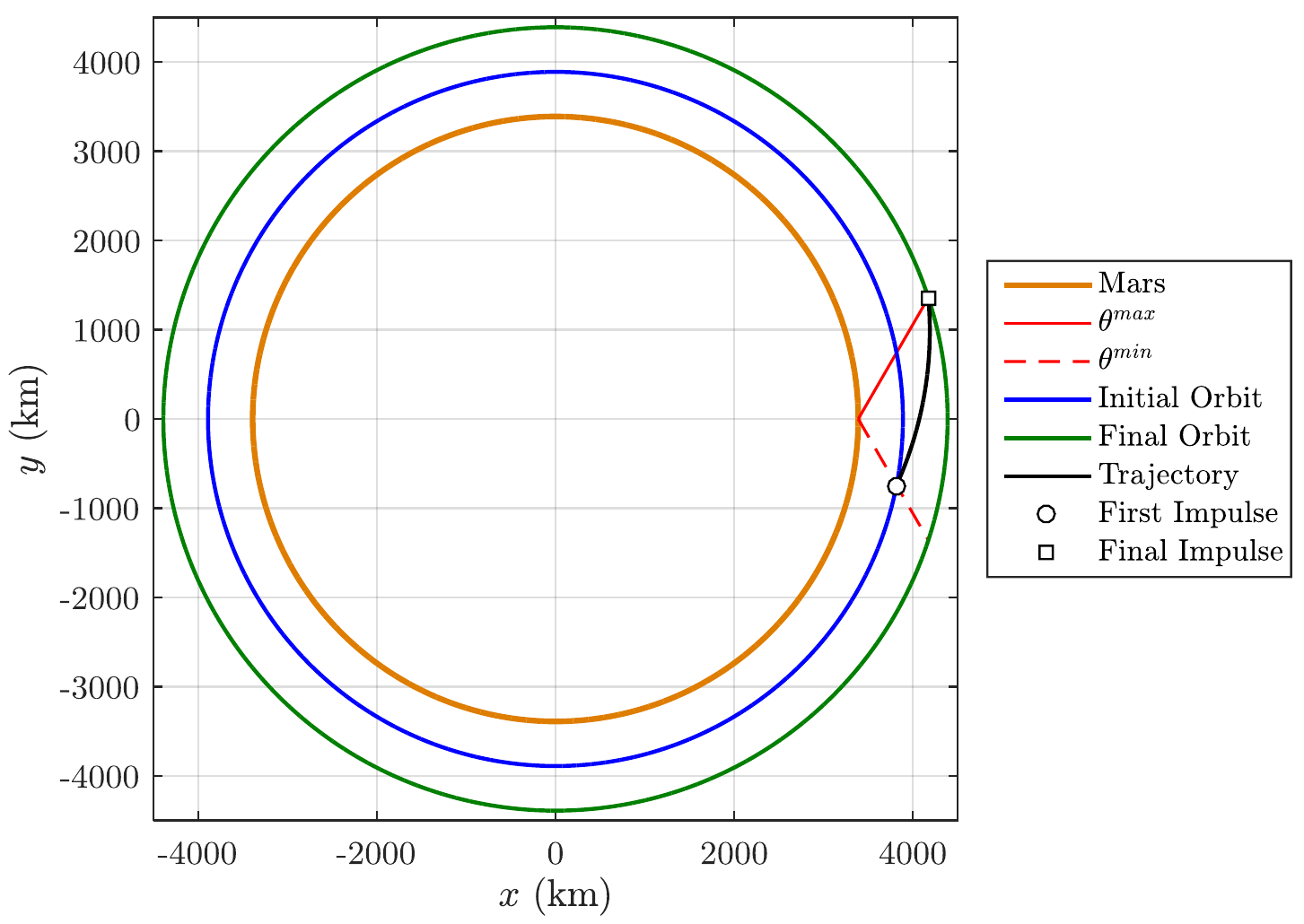}
\caption{Maneuver trajectory generated with $w_{CE}=1,w_{MI}=w_{V}=0$. The sum of impulses is $1.48$ km/s and the maximum impulse required is $0.78$ km/s.}
\label{fig:2}
\end{figure}

Fig. \ref{fig:3} shows the trajectory when the uncertainty is taken into account in which the first impulse position is postponed in order to provide more time for the estimation algorithm to converge. In this case, the sum of impulses is $2.37$ km/s and the maximum required impulse magnitude is $1.26$ km/s. Fig. \ref{fig:4} shows the case where the uncertainty has a major impact on the selection of impulse positions which is considered by increasing its cost function weight, $w_V$. In this case, the sum of impulses increased as high as $5.12$ km/s with a required maximum impulse of $2.68$ km/s. As is shown, when $w_{MI}$ is set to zero, the four-impulse trajectory reduces to a two-impulse trajectory ($\delta v_2=\delta v_3=0$). By increasing the value of $w_{MI}$ the maximum impulse required for the maneuver can be reduced where accordingly, Fig. \ref{fig:5} shows the optimal trajectory when the control effort and the maximum impulse are both important while no attentions are paid to the role of uncertainty. In this case, the sum of impulses is $1.91$ km/s while the maximum required impulse magnitude reaches a lower value of $0.5$ km/s. These amounts clearly show that how the uncertainty can be reduced by spending more energy. 

\begin{figure}[h]
\centering\includegraphics[width=0.7\linewidth]{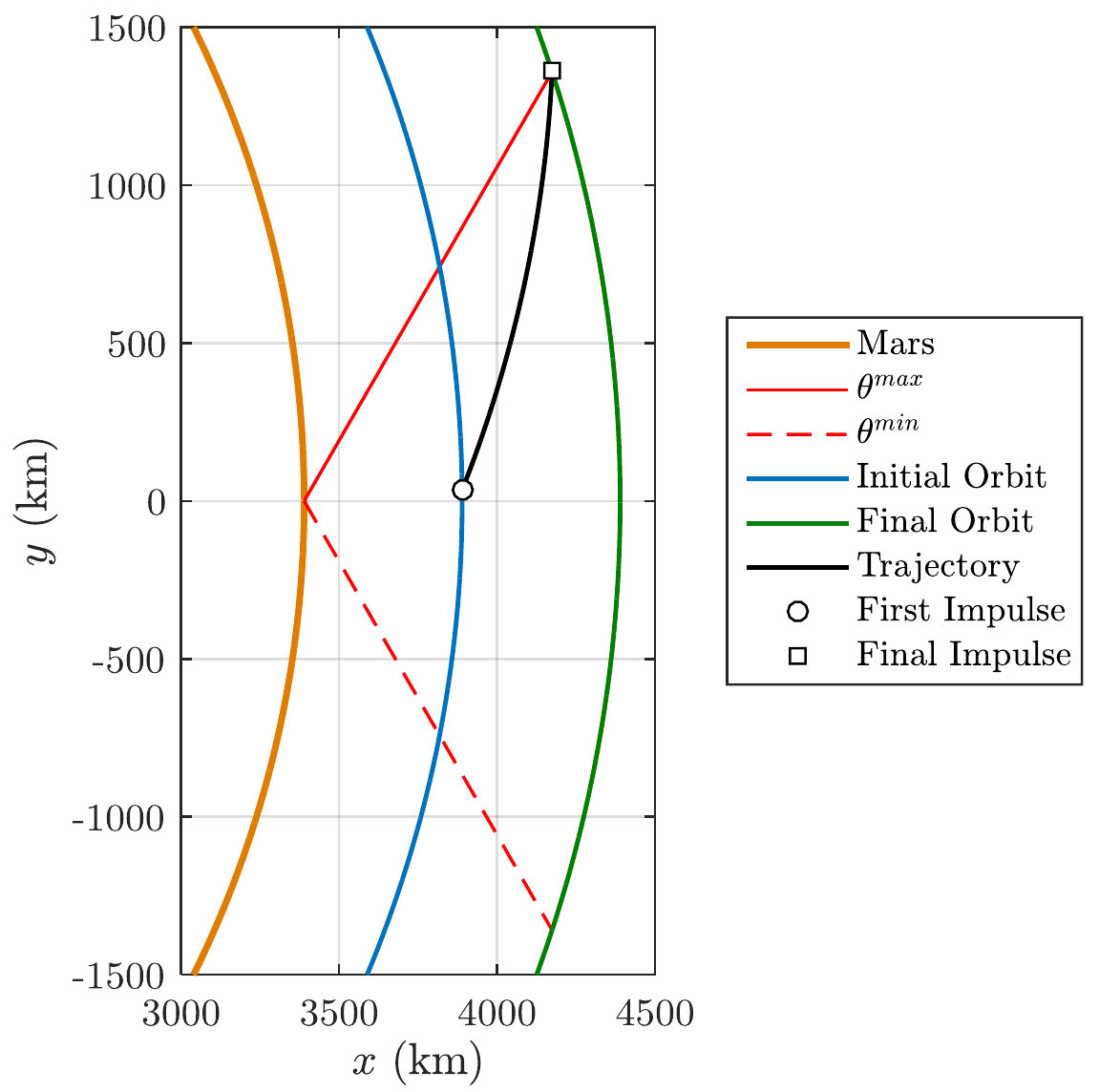}
\caption{Maneuver trajectory generated with $w_{CE}=1,w_{MI}=0,w_{V}=2$. The sum of impulses is $2.37$ km/s and the maximum impulse required is $1.26$ km/s.}
\label{fig:3}
\end{figure}

\begin{figure}[h]
\centering\includegraphics[width=0.7\linewidth]{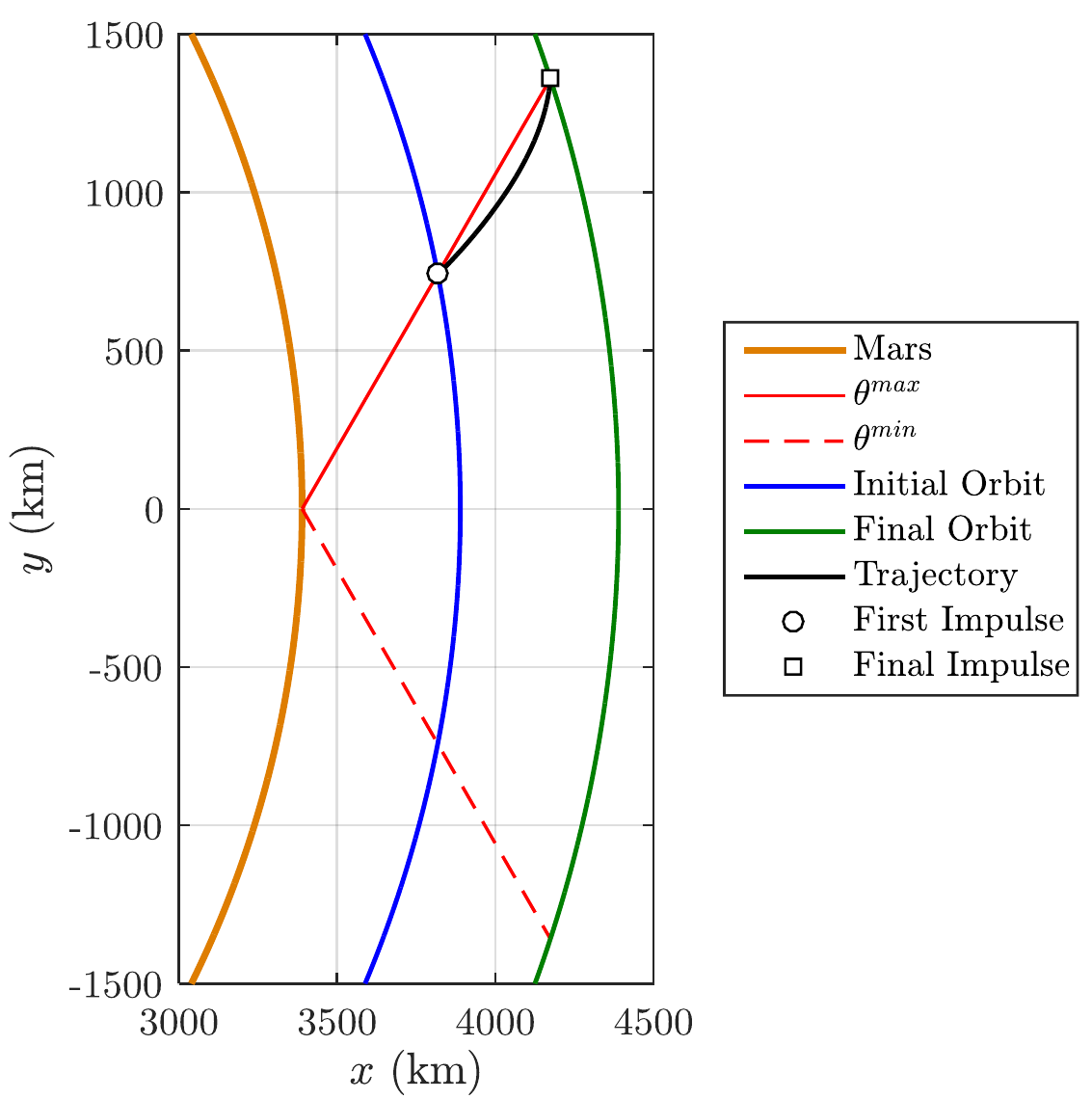}
\caption{Maneuver trajectory generated with $w_{CE}=1,w_{MI}=0,w_{V}=10$. The sum of impulses is $5.12$ km/s and the maximum impulse required is $2.68$ km/s.}
\label{fig:4}
\end{figure}

\begin{figure}[h]
\centering\includegraphics[width=0.7\linewidth]{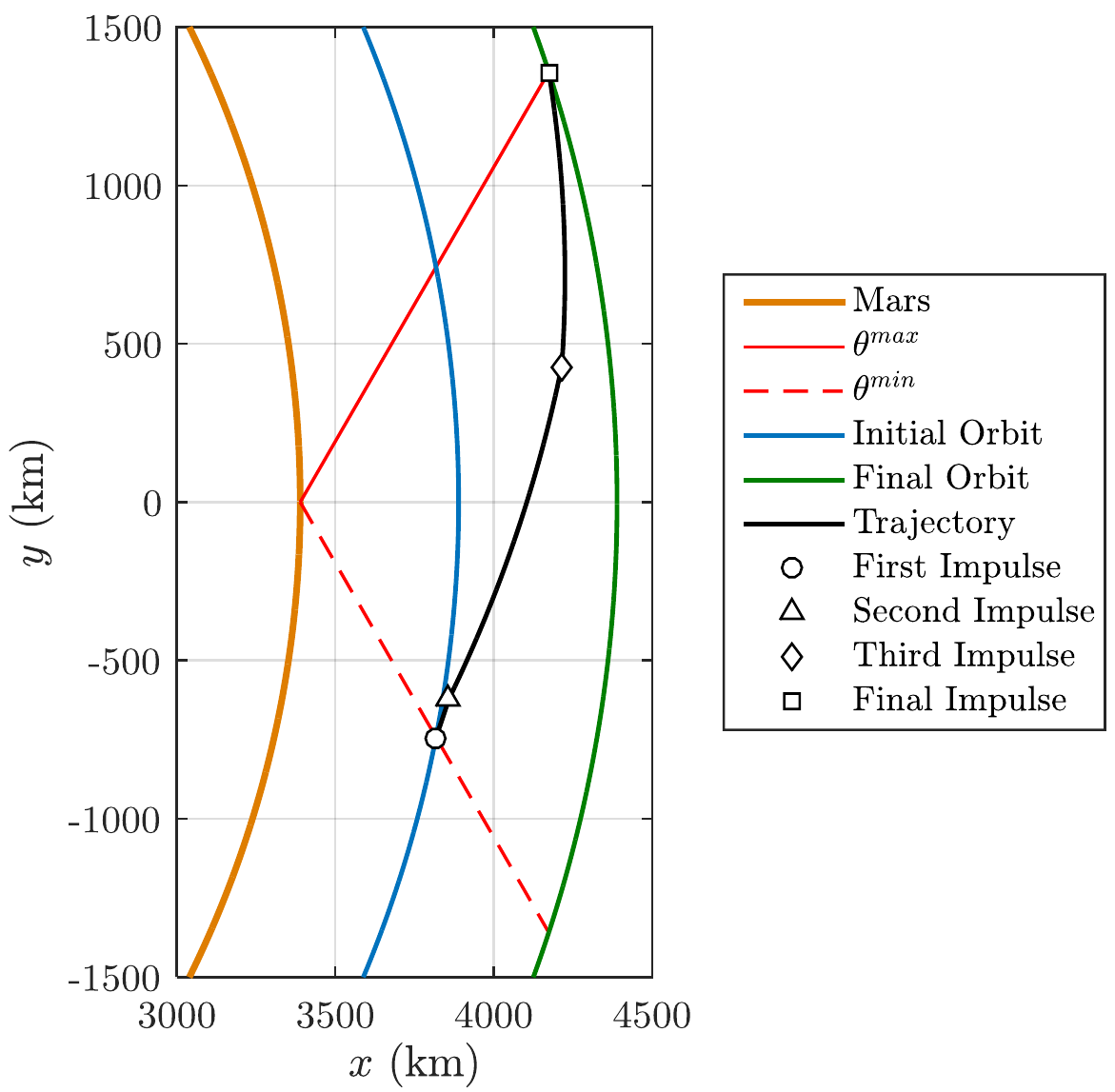}
\caption{Maneuver trajectory generated with $w_{CE}=1,w_{MI}=5,w_{V}=0$. The sum of impulses is $1.91$ km/s and the maximum impulse required is $0.50$ km/s.}
\label{fig:5}
\end{figure}

The general solutions, when all weighting values are non-zero, are shown in Figs. \ref{fig:6} and \ref{fig:7}. As is shown in Fig. \ref{fig:6}, selection of the weights may have other side effects, as collisions with the planet, which should be studied and devised by the designer. In the case study shown in Fig. \ref{fig:7}, the sum of impulses is $2.8$ km/s and the required maximum impulse magnitude is $0.71$ km/s. 

\begin{figure}[h]
\centering\includegraphics[width=1\linewidth]{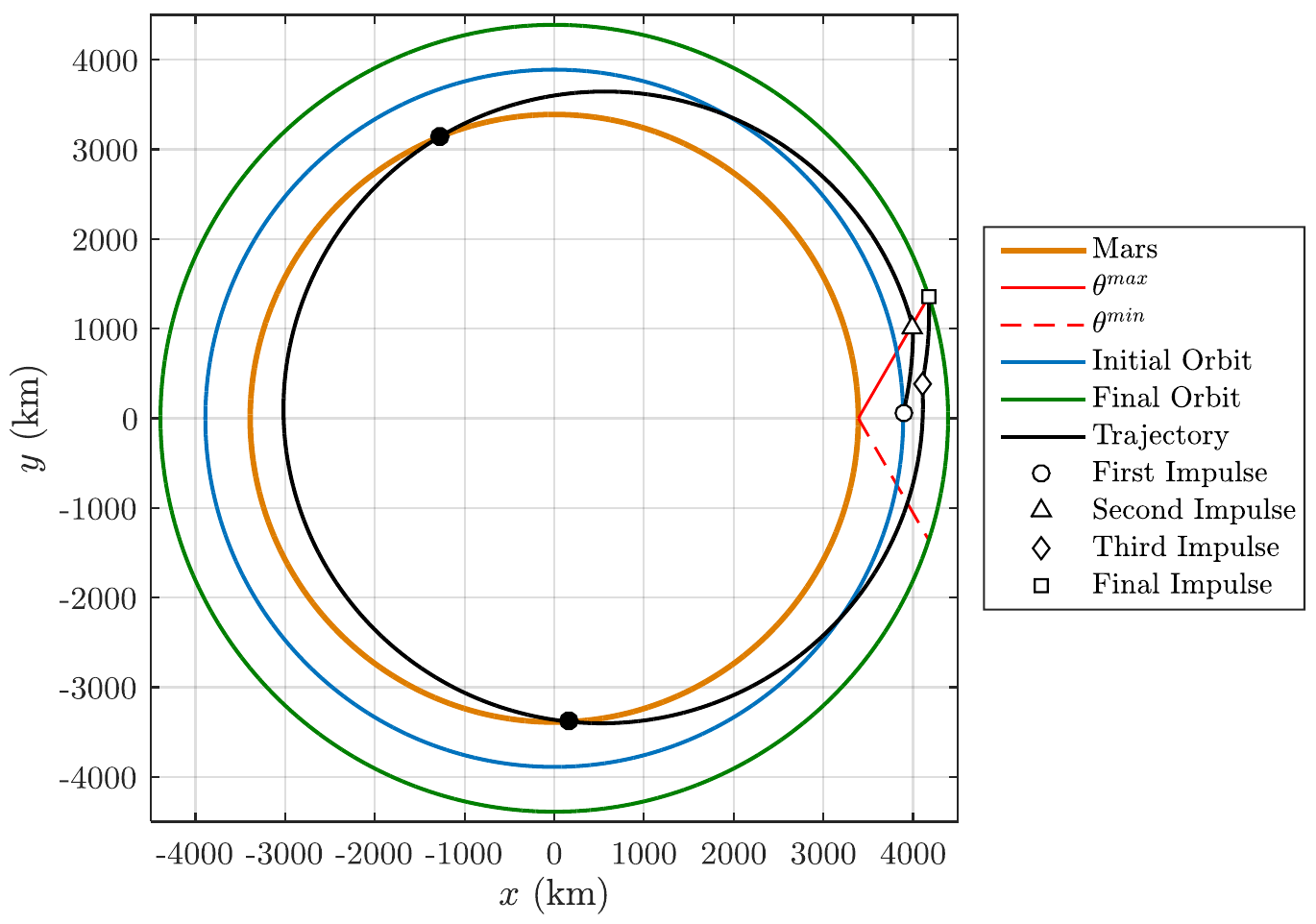}
\caption{Maneuver trajectory generated with $w_{CE}=1,w_{MI}=5,w_{V}=10$. The sum of impulses is $3.22$ km/s and the maximum impulse required is $0.83$ km/s. Collision points are shown by solid dots.}
\label{fig:6}
\end{figure}

\begin{figure}[h]
\centering\includegraphics[width=1\linewidth]{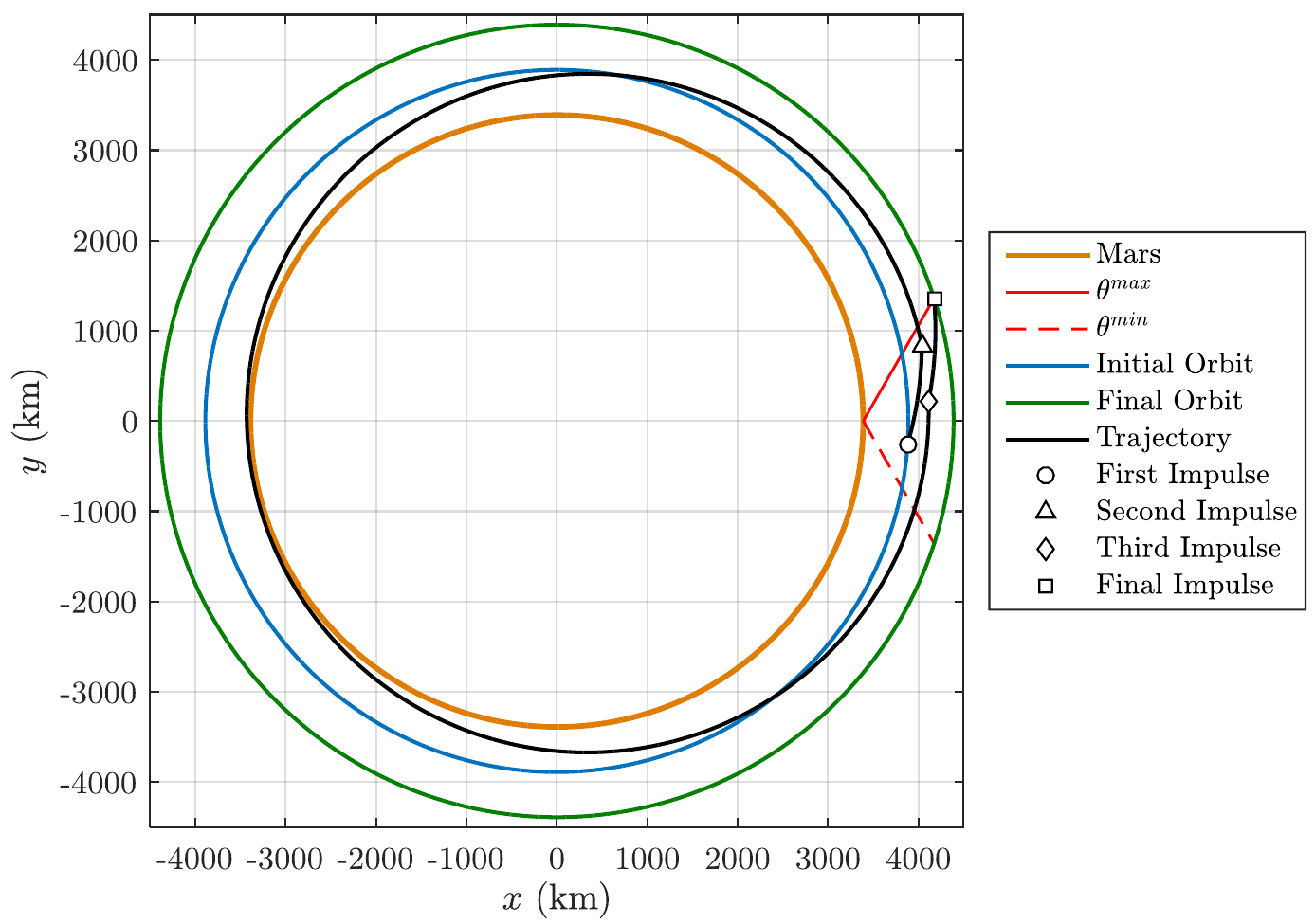}
\caption{Maneuver trajectory generated with $w_{CE}=1,w_{MI}=5,w_{V}=5$. The sum of impulses is $2.80$ km/s and the maximum impulse required is $0.71$ km/s.}
\label{fig:7}
\end{figure}

In Figs. \ref{fig:2}--\ref{fig:5}, Assumption \ref{ass:1} is satisfied. In case studies shown in Figs. \ref{fig:6} and \ref{fig:7}, Assumption \ref{ass:1} is violated and the trajectory turns around the celestial body. In the former cases the value of $\theta_2$ is less than $\theta_3$ while in the latter cases we have $\theta_2>\theta_3$. Regardless of the initial guesses used in Algorithm \ref{al:3}, the solutions are converged to the appropriate values which means that the proposed algorithms are robust to initial values without encountering any singularities. 

\section{Conclusions}
\label{S:5}

A multiple-impulse orbital maneuver (MIOM) scheme for preliminary trajectory optimization and mission design (MD) is proposed. The aforementioned problem is investigated while considering a limited observation window as well as the role of uncertainty involved in a realistic mission. The problem is formulated in a simple form so that a gradient-based optimization method can be implementable. The proposed MIOM approach is vital when the system lacks a global observation. Impulse positions and times have been considered as the optimization (design) variables in an actuated spacecraft dynamics model in which the Lambert's algorithm is incorporated for orbital maneuvers between arbitrary orbits in the three-dimensional space. A numerical case study is performed for MIOM under Mars gravitational field. The results showed how a trade-off can happen between the impulse time deferment (as a measure of uncertainty level), the control effort, and the maximum required impulse magnitude that should be considered in the MD. 

The future works may include a more realistic situation around a planet with different sources of observation which are provided from multiple space-based and/or ground-based stations. Solving the problem in a real operation field may needs more advanced and combined optimization techniques. Moreover, the same problem can rise in an asteroid environment with a highly perturbed gravity where a lander needs to be observed by a parent spacecraft. 


\bibliography{mybibfile}

\end{document}